\documentclass[11pt]{amsart}
\usepackage{amssymb, amscd, graphicx} 
\usepackage{pdfsync} 
\newtheorem{theorem}{Theorem}[section]
\newtheorem{proposition}[theorem]{Proposition}
\newtheorem{lemma}[theorem]{Lemma}
\newtheorem{corollary}[theorem]{Corollary}

\newtheorem{definition}[theorem]{Definition}
\newtheorem{question}[theorem]{Question}
\newtheorem{example}[theorem]{Example}
\newtheorem{remark}[theorem]{Remark}

\newtheorem*{lemma*}{Lemma}

\def\part#1{\bigbreak\noindent{\textsc{#1}}\medbreak}
\newcommand{\foot}{\setcounter{footnote}{0}\footnote}

\def\itm{\smallskip\noindent$\bullet$\ \,}
\def\qed{\hfill$\square$}

\newcommand{\itb}{\bfseries\itshape}

\newcommand{\tabletop}{\centering
\begin{tabular}{c c l c c c l c c c l} 
\hline \\ [-4ex]
&&&\vline &&&&\vline&&& \\
$n$ && $q$  &\vline & $n$ && $q$ &\vline & $n$ && $q$ \\ 
&&&\vline &&&&\vline&&& \\ [-1.3ex]
\hline \\ [-4ex]}

\def\bn{\mathbb N}
\def\bz{\mathbb Z}
\def\bq{\mathbb Q}
\def\br{\mathbb R}

\def\ba{\mathbb A}
\def\bo{\mathbb O}
\def\calp{\mathcal P}
\def\o{\frak o}
\def\oh{\widehat\o}

\def\rk{\textup{\rm rk\,}}
\def\coker{\textup{\rm coker}}
\def\lk{\ell k}
\def\br#1{\langle #1 \rangle}
\def\st{\ |\ }
\def\odd{\textup{\footnotesize odd }}
\def\medcirc{\raisebox{1.5pt}{$\scriptstyle\bigcirc$}}
\def\tor#1{\textup{Tor}_{#1}}
\def\mil#1{\mu_{#1}}
\def\mas#1{\omega_{#1}}
\def\les#1{\lambda_{#1}}
\def\lf#1{\phi_{#1}} 
\def\leg#1#2{(#1\,|\,#2)}
\def\res#1#2{[#1\,|\,#2]}

\def\mubar{\ov\mu}

\def\ov{\overline}

\begin{document}

\vskip-.3in
\vskip-.1in

\title{The Milnor degree of a $3$-manifold}
\author{Tim Cochran and Paul Melvin}
\date{}
\thanks{* Both authors are supported by grants from the National Science Foundation.}

\begin{abstract}
The Milnor degree of a 3-manifold is an invariant that records the maximum simplicity, in terms of higher order linking, of any link in the 3-sphere that can be surgered to give the manifold.  This invariant is investigated in the context of torsion linking forms, nilpotent quotients of the fundamental group, Massey products and quantum invariants, and the existence of 3-manifolds with any prescribed Milnor degree and first Betti number is established. 

Along the way, it is shown that the number $M_k^r$ of linearly independent Milnor invariants of degree $k$, for $r$-component links in the $3$-sphere whose lower degree invariants vanish, is positive except in the classically known cases (when $r=1$, and when $r=2$ with $k=2$, $4$ or $6$). 
\end{abstract}

\maketitle

\setcounter{section}{-1}

\vskip-.3in
\vskip-.1in

\parskip 4pt

\section{Introduction}  

This paper initiates a study of the Milnor degree, a 3-manifold invariant introduced by the authors in \cite{CM1}.  The definition is recalled and motivated below.  

All $3$-manifolds considered here will be closed, connected and oriented.  Any such manifold can be constructed by integral surgery on a framed link $L$ in the $3$-sphere $S^3$, written $S^3_L.$  Indeed there are infinitely many choices for the link $L$, and so in studying a given $3$-manifold, it is natural to seek the simplest ones.  But in what sense simplest?  One measure of the complexity of a link is its linking matrix, or more generally its set of higher order linking numbers, a.k.a.\ Milnor's $\bar\mu$-invariants \cite{M2}.  

So first define the {\itb\boldmath Milnor degree $\mu_L$ of a link $L$} in the $3$-sphere to be the degree of its first nonvanishing $\bar\mu$-invariant.  Here {\it degree} means {\it length minus one}, so the pairwise linking numbers are degree one.  If all of the $\bar\mu$-invariants vanish, as for knots or more generally boundary links, then the link is said to have infinite Milnor degree.  Thus the link invariant $\mu$ takes values in $\bn\cup\{\infty\}$.  Note that higher Milnor degree for a link indicates greater similarity with the unlink, and so in some sense greater simplicity.  For example the Hopf link has Milnor degree one, since the components have nonzero linking number, whereas the link obtained by Whitehead doubling both components of the Hopf link is a boundary link, and so has infinite Milnor degree.  Links with arbitrary finite degree can be constructed by repeatedly Bing doubling the Hopf link; see Figure 7 in \cite{M1}, and \S8 in \cite{C4}.  

Now define the {\itb\boldmath Milnor degree $\mil M$ of a $3$-manifold $M$} to be the supremum of the degrees of all possible links that can be surgered to give $M$,
$$
\mil{M} \ = \ \sup\,\{\mil{L} \st L\subset S^3 \textup{ with } M=S^3_L\} \ \in \ \bn\cup\{\infty\}.
$$

This complexity measure for 3-manifolds first arose in the authors' study of cyclotomic orders of quantum $SO(3)$-invariants at prime levels \cite{CM1}, and also appears as a measure of computational complexity for the quantum $SU(2)$-invariants at the fourth root of unity \cite{KM1}.  As it turns out, this connection with quantum topology provides a powerful tool for analyzing the Milnor degree.

As with other topological invariants defined in a similar fashion -- such as Heegaard genus or the surgery number -- the Milnor degree is hard to compute.  Indeed its value is unknown for many 3-manifolds, including even some lens spaces.  It will be seen however that in some situations, especially in the absence of homological torsion, classical techniques from algebraic topology can be brought to bear on this computation.   

In section 1 the manifolds of Milnor degree one are completely characterized in terms of their torsion linking forms.  In particular, the case when the first homology of the manifold is cyclic (homology lens spaces) is discussed in some detail.  In some circumstances one can also identify the manifolds of infinite degree in terms of their linking forms.  For example, it will be seen that if the manifold has prime power order first homology and is not of Milnor degree one, then it must have infinite Milnor degree.

In section 2, the Milnor degree of a 3-manifold with torsion free homology is related to the lower central series of its fundamental group, and consequently to its cohomological {\itb Massey degree}.  The Massey degree is known to be algorithmically computable, but it is difficult to calculate in practice.  Nevertheless, 3-manifolds with any given Massey degree are easily constructed and it follows that the Milnor degree assumes all values in $\bn\cup\{\infty\}$.   Building on the result mentioned in the abstract on the number of independent Milnor invariants of links (Lemma~\ref{lem:numbertheory}, proved in Appendix B), this section includes some realization results for the Milnor degree of manifolds with prescribed first homology.

It is a more difficult task to compute the Milnor degree in the presence of torsion, although Massey products can still be of some help. This problem is tackled in section 3 using quantum topology techniques.  This leads to a construction of rational homology spheres of arbitrary Milnor degree, and more generally, $3$-manifolds of arbitrary Milnor degree with any prescribed first Betti number.

Before embarking on a detailed analysis of the Milnor degree of $3$-manifolds, we make a few basic observations.  

\itm
The Milnor degree is invariant under change in orientation, that is
$
\mil{M} = \mil {\ov M}
$ 
where $\ov M$ is $M$ with the opposite orientation:  If $M$ is surgery on $L$, then $\ov M$ is surgery on the mirror image $\ov L$ with negated framings, and clearly $\mil {L}=\mil{\ov L}$. 

\itm 
The Milnor degree of a connected sums satisfies the inequality 
$
\mil {M\#M'} \ge \min(\mil{M},\mil{M'}).
$ 
For if $M$ and $M'$ are surgery on $L$ and $L'$, then $M\# M'$ is surgery on the split union $L\sqcup L'$, and clearly $\mil {L\sqcup L'} = \min(\mil{L},\mil{L'})$.  

Note that this inequality need not be an equality.  For example the connected sum $L(8,5)\#L(5,8)$ of lens spaces is $40$ surgery on the $(8,5)$-torus knot \cite{Mo}, and so of infinite degree, whereas both lens spaces are of degree 1, as will be seen in the next section.  As a consequence, realization results for the Milnor degree are subtler than one might first suspect; see Corollary~\ref{cor:connsum} and (the proof of) Theorem~\ref{thm:realization2} for upper bounds on $\mil{M\#M'}$.     

\itm There exist $3$-manifolds of infinite Milnor degree with any prescribed first homology $$\bz^r\oplus \bz_{n_1}\oplus\cdots\oplus\bz_{n_k},$$ for example
$
\#^rS^1\times S^2\,\#\,L(n_1,1)\,\#\,\cdots\,\#\,L(n_k,1),
$
obtained by surgery on an unlink with framings $0,\dots,0,n_1,\dots,n_k$.

\itm
All integral homology spheres have infinite Milnor degree.  This follows from the well known fact that they are all constructible by surgery on boundary links in the $3$-sphere, or it can be deduced from the following more general statement.

\itm
The Milnor degree is {\itb homological} in the sense that it can be defined using any integral homology sphere $\Sigma$ in place of $S^3$.  In other words $\mil M$ does not depend on which homology sphere $\Sigma$ is used as the base manifold, that is
$$
\mil{M} = \sup\{\mil{L} \st L\subset \Sigma \textup{ with } M=\Sigma_L\} \quad\textup{for {\it any} integral homology sphere $\Sigma$}
$$
where $\Sigma_L$ denotes the result of surgery on the framed link $L\subset \Sigma$, and $\mil{L}$ is the degree of the first non-vanishing $\bar\mu$-invariant of $L$ in $\Sigma$ (see section $2$ for a discussion of Milnor invariants in arbitrary integral homology spheres).   This is proved in Appendix A using the work of Habegger and Orr; cf.\ the proof of 6.1 in \cite{HO}.  

\vskip -.2in

\section{Manifolds of degree one}   

In this section, classical results from the theory of quadratic forms are used to characterize all 3-manifolds of Milnor degree one, and some of infinite degree, in terms of their torsion linking forms.

\part{Linking forms}

The {\itb linking form} of a $3$-manifold $M$ is the non-degenerate form 
$$
\lf M\colon \tor M\otimes \tor M \to \bq/\bz
$$
on the torsion subgroup $\tor M$ of $H_1(M)$ defined by $\lf M(a\otimes b) = \alpha\cdot\tau/n$, where $\alpha$ is any $1$-cycle representing $a$, and $\tau$ is any $2$-chain bounded by a positive integral multiple $n\beta$ of a 1-cycle $\beta$ representing $b$.  

If $M$ is surgery on a framed link $L$, then $\lf M$ is computed from the (integer) linking matrix $A$ of $L$, with framings on the diagonal, as follows.  First change basis (pre and post multiply by a unimodular matrix and its transpose) to transform $A$ into a block sum $\bo\oplus \ba$, where $\bo$ is a zero matrix and $\ba$ is {\itb nonsingular} (meaning invertible over $\bq$, i.e.\  having nonzero determinant).\foot{\,To see how this is done when $A$ is singular, start with a primitive vector $v$ in $\bz^n$ with $Av=0$, and complete this to a basis for $\bz^n$.  Using these basis vectors as the columns of a matrix $P$, we have $P^TAP = \bo_1\oplus \ba_1$ where $\bo_1$ is the $1\times 1$ zero matrix and null($\ba_1)=$ null($A)-1$.  The argument is completed by induction on the nullity of $A$.}  This corresponds to a sequence of handleslides in the Kirby calculus \cite{K}, transforming $L$ into $L_{\bo}\cup L_{\ba}$.  Now following Seifert \cite{Se}, the linking form $\lf M$ is presented by the (rational) matrix $\ba^{-1}$ with respect to the generators of $\tor M$ given by the meridians of the components of $L_{\ba}$.

The purely algebraic procedure just described associates with any symmetric integer matrix $A$ a non-degenerate linking form $\lf A$ on the torsion subgroup of $\coker(A)$, presented by $\ba^{-1}$ where $A \sim \bo\oplus\ba$ with $\ba$ nonsingular as above.  

It is a classical fact that if $A$ is nonsingular, then it can be recovered up to {\itb stable equivalence} from the isomorphism class of its linking form.  (Stable equivalence allows change of basis and block sums with diagonal matrices of $\pm1$'s; the former correspond to {\it handleslides} and the latter to {\it blow ups} in the Kirby calculus.)  This was proved by Kneser and Puppe \cite{kp} for the case when $\det(A)$ is odd, and in general by Durfee and Wilkens in their 1971 theses, later simplified by Wall \cite{Wa3} and Kneser (see Durfee \cite[Theorem 4.1]{Du}).  

For a singular matrix, one needs both its linking form and its nullity to recover its stable equivalence class.  This can be seen by changing basis to transform the matrix into the form $\bo\oplus\ba$ with $\ba$ nonsingular, as above, and then appealing to the nonsingular case.  Although this fact is presumably well-known, we have not been able to find a proof in the literature, and so credit it to ``folklore":

\begin{theorem} \label{thm:classical} \hskip-.1in {\textup{(Folklore)}} Two symmetric integer matrices are stably equivalent if and only if they have the same nullity and isomorphic linking forms.  
\end{theorem}

Needless to say, this theorem has a number of consequences regarding surgery presentations of $3$-manifolds, many of which are presumably known in some form to experts in the field: 
  
\begin{corollary} \label{cor:classical} Let $M$ be a $3$-manifold, and $A$ be a symmetric integer matrix with nullity equal to the first Betti number $r$ of $M$ and with linking form $\lf A \cong \lf M$.  Then $M$ can be constructed by surgery on a link with linking matrix $A$ in an integral homology sphere.
\end{corollary}

\vskip-.2in

\begin{proof}
Suppose that $M$ is surgery on a framed link $L'$ in $S^3$ with linking matrix $A'$.  Then $A'$ has nullity $r$ and $\lf A' \cong \lf M$, and so by the theorem, $A$ and $A'$ are stably equivalent.  Thus a basis change (as above) will transform $A'\oplus D'$ into $A\oplus D$ for suitable unimodular diagonal matrices $D$ and $D'$.  Letting $U'$ be an unlink far away from $L'$ with linking matrix $D'$ (so $M=S^3_{L'\cup U'}$), this means that $L'\cup U'$ can be transformed by handleslides to a link of the form $L\cup U$ with linking matrix $A\oplus D$.  Note that $U$ need not be an unlink, but it does have a unimodular linking matrix $D$, and so $N = S^3_U$ is a homology sphere containing the link $L$ with linking matrix $A$, and $M=N_L$, as desired. 
\end{proof}

Linking forms on finite abelian groups have been classified.  They decompose, albeit non-uniquely, as orthogonal sums of forms on cyclic groups and on certain non-cyclic 2-groups ~\cite{kk}\cite{Wa2}.  The form on the cyclic group $\bz_n$ with self-linking $q/n$ on a generator will be denoted by $(q/n)$.  Note that $q$ must be relatively prime to $n$, since the form is non-degenerate, and that 
$$
(q/n)\cong(q'/n) \ \ \iff \ \ q'\equiv k^2q\pmod n
$$
for some $k$ relatively prime to $n$.    

Any form isomorphic to $(\pm1/n)$ for some $n\ge1$ will be called {\itb simple}.  This includes the trivial form on $\bz_1=0$.   A direct sum of simple forms will be called {\itb semisimple}.  But beware, such forms might also be sums of non-simple or even non-semisimple forms, for example
$$
(1/40) \ \cong \ (2/5)\oplus(-3/8).
$$
In terms of their associated stable equivalence classes of symmetric integer matrices, semisimple forms correspond to diagonal matrices, and the simple forms correspond to diagonal matrices with at most one nonzero entry.  

Now observe that surgery on any {\itb diagonal framed link} in $S^3$ (meaning its linking matrix is diagonal, i.e.\ pairwise linking numbers vanish) can be performed in two stages:  first surger the sublink of $\pm1$-framed components, giving an integral homology sphere, and then surger the remaining components.  With this perspective, Corollary~\ref{cor:classical} has the following immediate consequence.   

\begin{proposition}\label{prop:forms} The linking form of a $3$-manifold $M$ is semisimple if and only if $M$ can be obtained by surgery on a diagonal link $L$ in an integral homology sphere, and is simple if and only there is such a link $L$ with at most one nonzero framing.  
\end{proposition}

By definition, a $3$-manifold is of degree greater than one if and only if it can be obtained by surgery on a diagonal link.  Hence the proposition yields the following characterization of manifolds of degree one.

\begin{corollary}\label{cor:deg1}
The $3$-manifolds of Milnor degree one are exactly those with non-semisimple linking forms.
\end{corollary}

Noting that the linking form of any $3$-manifold with torsion free homology is trivial, we deduce the following well-known result (see \cite[Lemma 5.1.1]{Les} for a direct proof):

\begin{corollary}\label{cor:torfree}
If $H_1M \cong \bz^r$, then $M$ has Milnor degree greater than one.  In fact $M$ can be obtained by zero surgery on an $r$-component diagonal link $L$ in an integral homology sphere $($by the last statement in Proposition~$\ref{prop:forms})$.
\end{corollary}

In particular if $r=0$ then $M$ is an integral homology sphere, and so as noted in the introduction is of infinite degree.  If $r=1$ then $M$ is surgery on a knot in a homology sphere, and so again is of infinite degree: 

\begin{corollary}\label{cor:z}
If $H_1M \cong \bz$ then $M$ has infinite Milnor degree.  
\end{corollary}

If $r\ge2$, then it will be shown in \S2 that $\mil M = \mil L$, and as consequence that the Milnor degree can assume any value greater than one (except $2$, $4$ and $6$ when $r=2$).  

The situation is more complex when $H_1M$ has {\sl torsion}.  In the simplest case when $H_1M$ is a finite cyclic group, i.e.\ $M$ is a homology lens space, the theorem gives the complete story in the extreme situations when $\lf M$ is either simple or non-semisimple.  But it tells us nothing when $\lf M$ is semisimple but not simple.  When can this happen?  This and related questions will be addressed below.

\break

\part{Homology lens spaces}

Assume that $H_1M \cong \bz_n$ for some $n\ge1$.  Then $\lf M\cong(q/n)$ for some $q$ relatively prime to $n$.  The application of Corollary~\ref{prop:forms} to this situation requires an understanding of when this form is (semi)simple.  

\begin{proposition}  $(a)$ The linking form $(q/n)$ on $\bz_n$ is simple if and only if $q$ is plus or minus a quadratic residue mod $n$.    

\smallskip
$(b)$ The linking form $(q/n)$ on $\bz_n$ is semisimple if and only if there is a factorization $n=n_1\cdots n_k$ with the $n_i$ pairwise relatively prime such that each form $(q_i/n_i)$ is simple, where $q_i = qn/n_i$.  
\end{proposition}  

\vskip-.1in
{\it Proof.} \ 
Criterion $(a)$ is just the definition of a simple form.  To prove (b), note that for any factorization $n=n_1\cdots n_k$ with the $n_i$ pairwise relatively prime, there exist integers $r_i$ for which 
$
q/n=r_1/n_1+\cdots+r_k/n_k,
$  
which yields an orthogonal splitting of the form 
$$
(q/n) \cong (r_1/n_1)\oplus\cdots\oplus(r_k/n_k).
$$
In fact any such splitting arises in this way.  Furthermore, the $r_i$ are uniquely determined mod $n_i$; indeed $r_i\equiv qs_i\pmod{n_i}$ where $s_i$ is any mod $n_i$ inverse of $n/n_i$.  Therefore by definition $(q/n)$ is semisimple if and only if $n$ has a factorization $n_1\cdots n_k$ as above for which each form $(qs_i/n_i)$ is simple.   But $q_i = qn/n_i \equiv qs_i(n/n_i)^2 \pmod{n_i}$, and so $(qs_i/n_i)\cong(q_i/n_i)$.  This completes the proof.  
\qed

Elementary number theoretic considerations show that $q$ (prime to $n$) is plus or minus a quadratic residue mod $n$, denoted $\res qn=1$ to evoke the Legendre symbol $\leg qp$, if and only if one of the following holds for all odd prime divisors $p$ of $n$:

\smallskip
\begin{itemize}
\item[$\bullet$] $\leg qp=1$ and $q\equiv1$ mod $\gcd(8,n)$ , or  
\smallskip
\item[$\bullet$] $\leg{-q}p=1$ and $q\equiv-1$ mod $\gcd(8,n)$
\end{itemize}

\smallskip\noindent
(see \cite[\S5.1]{IR}).  If $\res qn\ne1$, set $\res qn=-1$.  

\begin{example}\label{ex:primepowers} \textup{(Prime powers)
For $p$ prime, $\res q{p^e} = 1$ for {\it all} $q$ when either $p\equiv 3$ mod $4$ \ (since $\leg{-1}p=-1$) or $p=2$ and $e\le 2$.}  

\textup{If $p\equiv1$ mod $4$, then $\res q{p^e} =\leg qp$ (since $\leg{-1}p=1$) and so  $\res q{p^e} = 1 \iff q$ is a quadratic residue mod $p$.}

\textup{And if $e\ge3$ then $\res q{2^e}=1 \iff q\equiv\pm1\pmod8$.}
\end{example}

With this notation, the proposition says that $(q/n)$ is simple if and only if $\res qn=1$, and semisimple if and only if $n$ can be written as a product of ``$q$-quadratic" factors, defined as follows.

\begin{definition}\label{def:qquadratic}
\textup{A divisor $d$ of $n$ is {\itb\boldmath $q$-quadratic} if $d$ and $n/d$ are relatively prime and $\res{\,qd\,}{\,n/d\,}=1$.}
\end{definition}

In particular, $n$ is always $q$-quadratic, and $1$ is $q$-quadratic if and only if $\res qn=1$, i.e.\ $q$ is plus or minus a quadratic residue mod $n$.   In this language, Proposition~\ref{prop:forms} and Corollary \ref{cor:deg1} specialize to:

\begin{corollary}\label{cor:deg1lens}
Let $M$ be a $3$-manifold with first homology $\bz_n$ and linking form $(q/n)$.  Then $M$ has Milnor degree one if and only if $n$ is {\it not} a product of $q$-quadratic factors.  Furthermore, if $1$ is $q$-quadratic $($which for lens spaces just means that $M$ is homotopy equivalent to $L(n,1))$ then $M$ has infinite degree.  
\end{corollary}

It is straightforward (e.g.\ using Mathematica \cite{Wo}) to generate a complete list of non-semisimple forms $(q/n)$ for small $n$, as in Table 1:  For each $n$, the smallest residue values of $q$ are listed, one for each pair $(\pm q/n)$ of non-semisimple forms.  These correspond exactly to the lens spaces $L(n,q)$ of Milnor degree one.  

\begin{table}[!h]
\caption{\scshape Non-semisimple Cyclic Linking Forms $(\pm q/n)$} 
\tabletop
&&                             &\vline&&&                                &\vline&&&                                   \\ [1ex]
5   &&2                     &\vline&24  &&7                       &\vline&39  &&2, 5, 7                  \\ [1ex]
8   &&3                     &\vline&25  &&2, 3, 7              &\vline&40  &&7, 11, 19              \\ [1ex]
13  &&2, 5               &\vline&29  &&2, 3, 8, 12        &\vline&41  &&3, 6, 11, 12, 13   \\ [1ex]
16  &&3                   &\vline&32  &&3, 5                   &\vline&45  &&2, 7, 8                   \\ [1ex]
17  &&3, 5               &\vline&34  &&3, 5                   &\vline&48  &&7, 17                     \\ [1ex]
20  &&3                   &\vline&37  &&2, 5, 6, 8, 13    &\vline&52  &&5, 7, 11                 \\ [1ex]
&&& \vline &&&& \vline &&& \\ [-1.4ex]
\hline
\end{tabular}
\end{table}

The natural numbers $n$ for which $\bz_n$ supports a non-semisimple form, the first few of which appear in the table, will be called {\itb linked} numbers.  All other natural numbers will be called {\itb unlinked}.  Alternatively, these notions can be phrased in terms of the following:

\begin{definition}\label{def:milnor}
\textup{For any finitely generated abelian group $A$, define the {\itb Milnor set} $\mil A$ of $A$ to be the set of all natural numbers that can be realized as Milnor degrees of $3$-manifolds with first homology $A$,
$$
\mil A \ = \ \{\mil M \st M \textup{ is a $3$-manifold with } H_1M \cong A\}\ \cap\ \bn.
$$
Infinity is excluded because it can always be realized, as noted in the introduction.}
\end{definition}

Now Corollary~\ref{cor:deg1} shows that $n$ is linked if and only if $1\in\mil{\bz_n}$ (meaning there exist \break $3$-manifolds $M$ of degree one with $H_1M \cong \bz_n$).   

For example, from the calculations in Example~\ref{ex:primepowers} it follows that the linked prime powers are exactly the $p$-powers for $p\equiv1\pmod4$ and the $2$-powers $\ge8$, and that any product of primes all congruent to $3\pmod4$ is unlinked.  Furthermore, it is clear from the definitions that if $n$ is a prime power, then every semisimple form on $\bz_n$ is simple, and so as a consequence:

\begin{corollary}
If $n$ is a prime power $p^e$, then 
$$
\mil{\bz_n} \ = \ 
\begin{cases}
\{1\} &\text{if $p\equiv1$ (mod $4$) or $p=2$ and $e\ge3$} \\
\ \, \emptyset &\text{ otherwise}.
\end{cases}
$$
\end{corollary}

The preceding discussion has brought attention to the natural numbers $n$ for which every semisimple form on $\bz_n$ is simple.  Such numbers will be called {\itb quasiprime}, since as noted above they include all prime powers.  The first few non-quasiprimes are $10$, $12$, $15$, $21$ and $24$.  By Corollary~\ref{cor:deg1}, we have $\mil{\bz_n} =\{1\}$ or $\emptyset$ for any quasiprime $n$, according to whether $n$ is linked or not.  

Pinning down the Milnor degree of non-simple manifolds $M$ with cyclic first homology of non-quasiprime order is much more difficult.  In fact we do not at present have a finite upper bound for the Milnor degree of any such manifold; conceivably they all have infinite degree. 

\begin{question}
Do there exist $3$-manifolds with finite cyclic first homology of finite Milnor degree greater than one?
\end{question}  

Of course {\sl lower} bounds for the Milnor degree can be established by displaying suitable surgery links.   For example $L(10,3)$, the smallest non-simple quasiprime lens space, has degree at least $3$.  Indeed, it can be obtained by surgery on a two-component link of degree $3$.  This is shown in Figure~\ref{fig:lens} using the Kirby calculus \cite{K}, starting with $10/3$ Dehn surgery on the unknot.  Unfortunately, that is all that we currently know about its Milnor degree.  Note however that manifolds of infinite degree with the same linking form are easily constructed, e.g.\ surgery on the unlink with framings $2$ and $-5$.

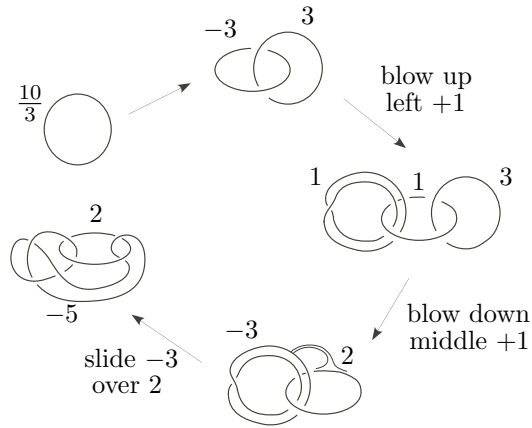
\begin{figure}[h!]\label{fig:lens}
\setlength{\unitlength}{1pt}
\begin{picture}(200,165)
\put(0,0){\includegraphics[height=150pt]{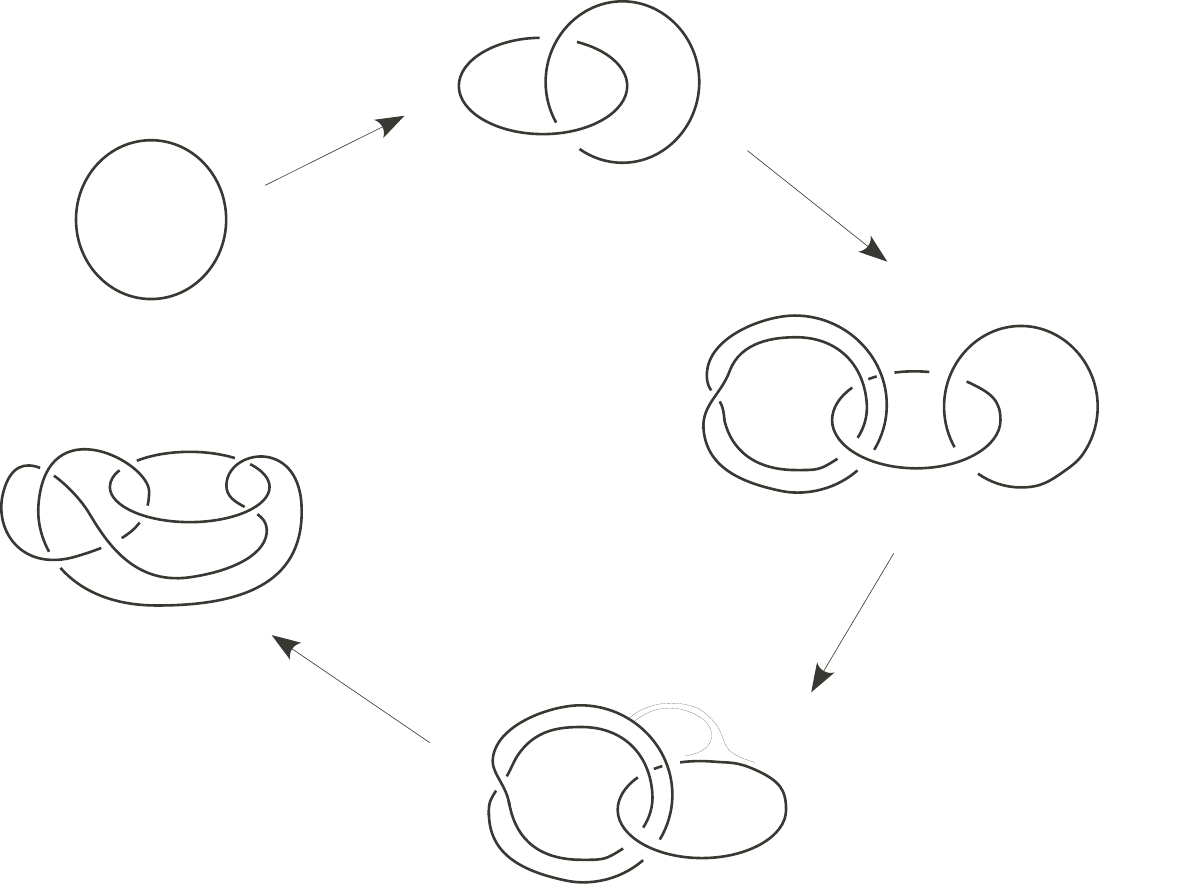}}
\put(2,120){$\frac{10}3$}   
\put(73,148){\small$-3$}  \put(110,153){\small3} 
\put(30,78){\small2}  \put(13,40){\small$-5$} 
\put(140,130){\small blow up} \put(142,120){\small left $+1$}  
\put(113,92){\small1}  \put(152,90){\small1}  \put(185,92){\small3}
\put(150,40){\small blow down} \put(151,30){\small middle $+1$}
\put(81,34){\small$-3$}  \put(125,24){\small2} 
\put(28,23){\small slide $-3$}  \put(31,13){\small over $2$}  
\end{picture}
\caption{Surgery descriptions for $L(10,3)$}
\end{figure}

\section{Manifolds with torsion free homology}   
In this section, the Milnor degree of any 3-manifold with torsion-free homology is related to the lower central series of its fundamental group, and thence to the Massey products of its one-dimensional cohomology classes.  As a result the degree of such a manifold can in principal be computed.  Some aspects of this theory that hold in the presence of torsion will be discussed at the end of the section.  

We begin by reviewing the definition and some basic properties of Milnor's $\mubar$-invariants for links in the $3$-sphere~\cite{M2}, or more generally in any integral homology sphere (see e.g.~\cite{CO3}, \cite[\S2]{FT2}, \cite[Appendix A]{HO}).  By considering only the first nonvanishing  invariants, we avoid any consideration of indeterminacy.  Throughout, the lower central series of a group $G$ will be written $G = G_1\rhd G_2 \rhd G_3 \rhd \cdots$ where $G_k=[G_{k-1},G]$ for $k>1$. 
 
\part{Milnor's link invariants}

Let $L$ be an $r$-component ordered oriented link in an integral homology sphere $\Sigma$, and set $G=\pi_1(\Sigma-L)$.   A presentation of the nilpotent quotient $G/G_k$ for each $k$ can be obtained as follows, generalizing a result of Milnor (Theorem 4 in \cite{M2}).  

Enlarge $L$ to a connected $1$-complex $\widehat L$ by adjoining disjoint paths from its components to a common basepoint, and choose based meridians $m_1,\dots,m_r$ and longitudes $\ell_1,\dots,\ell_r$ in $\Sigma-\widehat L$ for the components of $L$.  Let $F$ be the free group of rank $r$ generated by the $m_i$, and set $\widehat G = \pi_1(\Sigma-\widehat L)$.  Then there is a commutative diagram
$$
\begin{CD}
F            @>\phi>>        \widehat G       @>h>>     G             \\
@VVV                            @VVV                                 @VVV    \\
F/F_k            @>\phi_k>>        \widehat G/\widehat G_k       @>h_k>>     G/G_k            
\end{CD}
$$
for each $k\ge1$, where $\phi$ is the ``meridional" map sending $m_i$ to itself, $h$ is induced by the inclusion $\Sigma-\widehat L \hookrightarrow \Sigma-L$, and the vertical maps are the natural projections.   

Observe that $\phi$ is $2$-connected on integral homology (it is clearly an isomorphism on $H_1$, and $H_2F=H_2\widehat G=0$) and so $\phi_k$ is an isomorphism by Stalling's theorem ~\cite[Theorem 3.4]{St}.  In particular $\phi_k$ is surjective, so for each longitude $\ell_i$ we can choose an element $\ell_i^k\in F$ that represents $\ell_i$ in the sense that
$$
\phi(\ell_i^k)\equiv\ell_i\pmod{\widehat G_k}.
$$
The particular choice of $\ell_i^k$ will not affect our subsequent discussion.  Such elements will be called {\itb Milnor words} of {\it degree} $k-1$ (or {\it length} $k$) for $\ell_i$.  They form a coset of $F_k$ in $F$, since $\phi_k$ is injective.  

Also observe that $h$ is surjective with kernel normally generated by the commutators $[m_i,\ell_i]$ (for $i=1,\dots,r$) since $\Sigma-L$ is obtained from $\Sigma-\widehat L$ by adding $2$-cells along these commutators (and one $3$-cell).   It follows that the composition $h_k\circ\phi_k$ is an epimorphism with kernel normally generated by the cosets of the commutators $[m_i,\ell_i^k]$, and so
$$
G/G_k\ \cong \ \langle m_1,\dots,m_r \st F_k,\,[m_i,\ell_i^k] \ \textup{for}\ 1\leq i\leq r\rangle.
$$

Now by definition, Milnor's invariants of degree $0$ are zero.  Assuming inductively that the invariants of degree less than $k$ vanish, those of degree $k$ (or equivalently length $k+1$) are defined as follows:  

For any sequence $I=i_1\dots i_ki$ of $k+1$ integers between $1$ and $r$, the integer invariant $\mubar_L(I)$ is the coefficient of $h_{i_1}\cdots h_{i_k}$ in $e(\ell_{i}^{k+1})$, where $e$ is the Magnus embedding $m_i\mapsto 1+h_i$ of $F$ into the group of units in the ring of  power series in noncommuting variables $h_1,\dots,h_r$.   The Magnus embedding has the property that $x\in F_{k+1}$ if and only if $e(x)$ is of the form $1+h$ where $h$ has only terms of degree $>k$.  Thus fixing $i$ but letting $i_1,\dots,i_{k}$ vary, the collection of integers $\mubar_L(I)$ are all zero precisely when $\ell_{i}^{k+1}$ lies in $F_{k+1}$.  Allowing $i$ to vary, this implies

\begin{lemma}\label{lem:milnor1}
Let $L$ be an $r$-component link with Milnor words $\ell_i^k$ in the free group $F$ of rank $r$, as above.  Then 
$\mil L = \sup\,\{k \st \ell_i^k\in F_k \textup{ for all } i\}$.
\end{lemma}

\begin{remark}
\textup{With a little work using the presentation of $G/G_k$ and properties of the Magnus embedding, it can be shown that
$$
\mil L \ = \ \sup\,\{k \st \ell_i\in G_k \textup{ for all } i\} \ = \ \sup\,\{k \st F/F_k\cong G/G_k\}-1
$$
although this will not be needed below.}
\end{remark}

\part{Zero surgery and nilpotent quotients}

Any 3-manifold $M$ whose first homology is $\bz^r$ can be obtained by zero surgery on an \break $r$-component {\it diagonal} link -- meaning pairwise linking numbers vanish -- in some homology sphere, as noted in Corollary~\ref{cor:torfree} (cf.\,\cite[Lemma 5.1.1]{Les}).  The main result of this section is that the Milnor degree of $M$ can be calculated from \emph{any} such framed link description. Thus for manifolds with torsion-free homology, the Milnor degree can theoretically be computed.  This is in sharp contrast to the situation when torsion is present.

\begin{theorem}\label{thm:main} 
If $M$ is zero surgery on a diagonal link $L$ in an integral homology sphere, then $\mil{M}=\mil{L}$.
\end{theorem}

Before giving the proof, we derive the following consequence, the first of several realization results that will be established.  Recall that $\mil A$ denotes the set of all natural numbers that arise as Milnor degrees of $3$-manifolds with first homology $A$.

\begin{corollary}\label{cor:realization1} 
$$
\mil{\bz^r} = \begin{cases}
\emptyset &\text{if $r=0$ or $1$} \\
\bn-\{1,2,4,6\} &\text{if $r=2$} \\
\bn-\{1\} &\text{if $r\ge3$}
\end{cases}
$$
\end{corollary}

\noindent{\it Proof}. \ 
Let $\mil r$ denote the set of all natural numbers that arise as the Milnor degrees of \break $r$-component diagonal links,\foot{Note that $1\notin\mil r$ since we are restricting to diagonal links.  Also, we exclude $\infty$ from $\mil r$ since it is realized for any $r$ by the $r$-component unlink.} so $\mu_{\bz^r} = \mu_r$ by the theorem.  Clearly $\mil0 = \mil1=\emptyset$, and 
$$
\mil r = \bn - \{k \st M_k^r = 0\}
$$
for $r\ge2$, where $M_k^r$ is the number of linearly independent Milnor invariants of degree $k$ distinguishing $r$-component links in the $3$-sphere whose lower degree invariants vanish.  

The {\itb Milnor numbers} $M_k^r$ were computed by Orr \cite{O1} to be
$$
M_k^r = rN_k^r - N_{k+1}^r
$$
where $N_k^r$ denotes the number of basic commutators of length $k$ in $r$ variables, given classically by Witt's formula
$$
N_k^r \ = \ \frac1k \, \sum_{d|k} \, \mu(d) \, r^{k/d}.
$$
Here $\mu(d)$ is the M\"obius function, defined to be $+1$ if $d=1$ or $d$ is a product of an even number of distinct primes, to be $-1$ if $d$ is a product of an odd number of distinct primes, and to be $0$ otherwise.  Therefore the corollary reduces to the following number theoretic result, whose proof is deferred to Appendix B. 

\begin{lemma}\label{lem:numbertheory}
The Milnor number $M_k^r $ is positive, or equivalently $N_{k+1}^r < rN_k^r$, for all integers $r,k\ge2$
except when $r=2$ and $k=2$, $4$ or $6$.  
\end{lemma}

\begin{remark} \textup{(a) A little more can be said for $3$-manifolds $M$ with $H_1M\cong\bz^2$, namely that $\mil M=3\iff$ the Lescop invariant $\les M$ of $M$ \cite{Les} is nonzero.  Indeed $\les M$ is equal to the negative of the
Sato-Levine invariant of any $2$-component link whose zero surgery produces $M$ ~\cite[Prop.\ T5.2]{Les}, which in turn equals $-\mubar(1122)$ of the link ~\cite[Th.\  9.1]{C3}, the unique (up to sign) $\mubar$-invariant of degree $3$ ~\cite[App.\ B]{C4}\cite{O1}.}

\smallskip

\textup{(b) Our proof of Corollary~\ref{cor:realization1} (via Lemma~\ref{lem:numbertheory}, proved in the appendix) is non-constructive.  Using the techniques of \cite{C4}, however, one can give explicit examples for each $r\ge3$ of $r$-component links of any given Milnor degree $d\ge2$, and thus by doing zero surgery on these links, of $3$-manifolds with first homology $\bz^r$ of Milnor degree $d$.}

\textup{For example, one such link is the split union $L_d^r$ of the $3$-component link $L_d$ in Figure~\ref{fig:realize} with the $(r-3)$-component unlink.  Note that $L_d$ is obtained from the $(d-1)^{\textup{st}}$ iterated Bing-double of the Hopf link (denoted $H^d$ in \S3 below) by banding together some of its components, following the procedure of ~\cite[\S7.4]{C4}.  By ~\cite[\S6]{C4}, the Milnor invariant $\mubar(32...21)$ of degree $d$ is equal to the single self-linking number $\lk(c(3),c(2...21))=\pm\lk(c(2...23),(2...21))$. For $L_d$, this linking number is equal to $1$ while the invariants of degree less than $d$ vanish (see 7.2 and 7.4 in \cite{C4}). Therefore $\mil {L^r_d}=d$.}
\end{remark}

\vskip .2in

\begin{figure}[h!]\label{fig:realize}
\setlength{\unitlength}{1pt}
\begin{picture}(100,100)
\put(-35,10){\includegraphics[height=100pt]{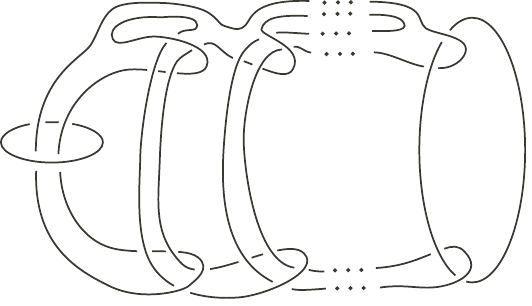}}
\put(-44,60){\small1}
\put(140,40){\small3}
\put(-25,85){\small2}
\put(45,0){\small $d-1$ loops}
\end{picture}
\caption{A $3$-component link $L_d$ of degree $d$}
\end{figure}
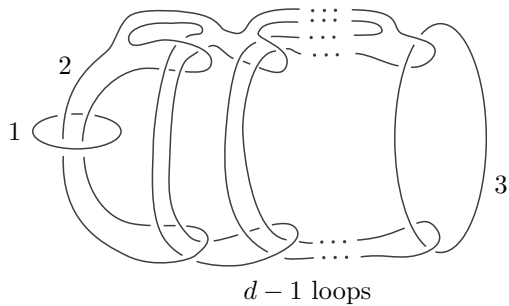

For $r=2$, Milnor \cite[Fig.\,1]{M2} has given examples (without proof) in each odd degree $d$, shown below in Figure~\ref{fig:realize2}.  It was confirmed in \cite[Example 2.7]{C4} that these do indeed have degree $d$.  It should be feasible using the same methods to produce such examples for even $d$ as well, although we have not done so here.

\begin{figure}[h!]
\setlength{\unitlength}{1pt}
\begin{picture}(230,90)
\put(0,0){\includegraphics[height=70pt]{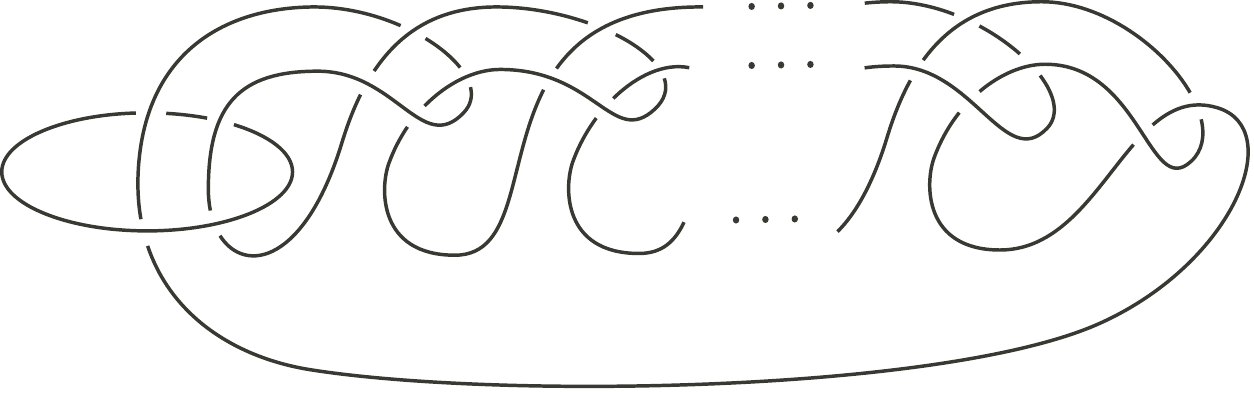}}
\put(0,50){\small1}
\put(210,10){\small2}
\put(100,80){\large $\frac{d-1}2$ \small loops}
\end{picture}
\caption{Milnor's $2$-component links of odd degree $d$}
\label{fig:realize2}\end{figure}
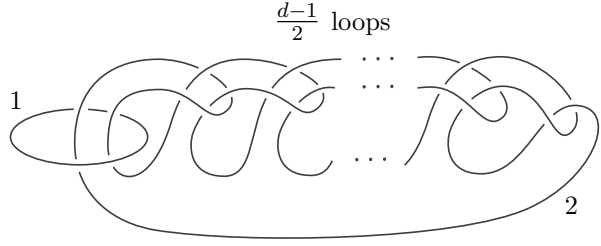

We now turn to the proof of Theorem~\ref{thm:main}.  It is based on the following characterization of the Milnor degree of a diagonal link:

\begin{lemma}\label{lem:milnor2}
Let $L$ be an $r$-component diagonal link in an integral homology sphere $\Sigma$, $M$ be the $3$-manifold obtained by zero surgery on $L$, $\pi$ be the fundamental group of $M$, and $F$ be a free group of rank $r$.  Then 
$
\mil L = \sup\,\{k \st F/F_k\cong\pi/\pi_k\}.
$
\end{lemma}

\vskip-.1in

\begin{proof}
Set $G = \pi_1(\Sigma-L)$.  Then as shown above
$$
G/G_k\ \cong \ \langle m_1,\dots,m_r \st F_k,\,[m_i,\ell_i^k] \ \textup{for}\ 1\leq i\leq r\rangle
$$
for all $k$, where $\ell_i^k$ are the associated Milnor words.  Zero surgery adds the relations $\ell_i=1$, and so for each $k$
$$
\pi/\pi_k\ \cong \ \langle m_1,...,m_m \st F_k,\,\ell_i^k \ \textup{for}\ 1\leq i\leq r\rangle,
$$
evidently a quotient of $F/F_k$ by the $\ell_i^k$.  It is clear from this presentation that if $\ell_i^k\in F_k$ for all $i$, then $\pi/\pi_k \cong F/F_k$.  Conversely, if $\pi/\pi_k \cong F/F_k$, then the quotient map above will be an isomorphism (since $F/F_k$ is nilpotent and hence Hopfian, as are all finitely generated nilpotent groups \cite[Theorem 5.5]{MKS}),  and so $\ell_i^k\in F_k$ for all $i$.  Therefore 
$$
\{k \st \ell_i^k\in F_k \textup{ for all } i\} \ = \ \{k \st F/F_k\cong\pi/\pi_k\}
$$
and the formula for $\mil L$ follows from Lemma~\ref{lem:milnor1}. 
\end{proof}

\begin{proof}[of Theorem~$\ref{thm:main}$]
By definition $\mil M\ge\mil L$, so we must prove $\mil M\le\mil L$.  Let $J$ be any framed link with $\mil J\ge\mil L$ whose surgery produces $M$.  It suffices to show $\mil L \ge \mil J$.    Since $L$ is a diagonal link, $J$ must be as well.  Let $J_0$ be the zero-framed sublink of $J$.  Then $\mil{J_0}\ge\mil J$, and $M$ is surgery on $J_0$ in the homology sphere obtained by surgery on $J-J_0$ (all of whose framings are $\pm1$ since $H_1M$ is torsion free).   But Lemma~\ref{lem:milnor2} shows that $\mil L$ is characterized by a property of the fundamental group of $M$, and so will be the same for any other link whose zero surgery produces $M$, such as $J_0$.  Therefore $\mil L = \mil{J_0} \ge \mil J$ as desired.
\end{proof}

The following characterization of the Milnor degree of a $3$-manifold with torsion free homology in terms of the lower central series of its fundamental group is an immediate consequence of Lemma~\ref{lem:milnor2} and Theorem~\ref{thm:main}.

\begin{theorem}  Let $M$ be a $3$-manifold with first homology $\bz^r$ and $F$ be a free group of rank $r$.  Then $\mil M \ = \ \sup\,\{k \st \pi_1M/(\pi_1M)_k \cong F/F_k\}$. \hfill$\square$
\end{theorem}

By Proposition 6.8 in \cite{CGO}, this lower central series condition is also a characterization of the {\itb Massey degree} $\mas M$ of $M$, defined to be the length $k$ of the first non-vanishing Massey product $\langle x_1,\dots,x_k\rangle$ of $1$-dimensional cohomology classes in $M$, or $\infty$ if all such Massey products vanish:

\begin{corollary} If $M$ is a $3$-manifold with torsion free homology, then its Milnor degree $\mil M$ is equal to its Massey degree $\mas M$. \hfill$\square$
\end{corollary}

\vskip-.2in

\begin{remark}
\textup{The Milnor degree is also related to the notion of {\it $n$-surgery equivalence} \cite{CGO}, the equivalence relation $\sim_n$ on $3$-manifolds $M$ generated by $\pm1$-framed surgery on links whose components lie in $(\pi_1M)_{n}$ (for $n\ge2$).  In particular, Theorem 6.10 in \cite{CGO} states that if $H_1M\cong\bz^r$, then $M \sim_n \#^r S^1\times S^2$ if and only if $\mas M\ge 2n-1$.  It follows that
$$
\lceil\textstyle\frac12 \mil M\rceil = \sup\,\{n \st M \sim_n \#^r S^1\times S^2\}.
$$}
\end{remark}

\vskip-.2in
\vskip-.2in

\part{Modifications in the presence of torsion}

Some of the ideas used in the proof of Theorem~\ref{thm:main} carry over in the presence of torsion, and can be used to extend the realization result, Corollary\ref{cor:realization1} (this result can also be deduced easily from \cite[Theorem 7.3]{St}). 

\begin{theorem}\label{thm:mainextension} Let $M$ be a $3$-manifold
with first Betti number $r$ and $F$ be a free group of rank $r$.   Set $\pi = \pi_1M$.  Then any map $\theta:F\to \pi$ that induces an isomorphism on $H_1/${\rm Torsion} induces a monomorphism
$$
\theta_n:F/F_k\to \pi/\pi_k
$$
for every $k$ not exceeding the Milnor invariant $\mil M$.
\end{theorem}

\vskip-.2in

\begin{proof} The conclusion is trivially true for $k=1$, so assume that $\mil{M}\ge k>1$.  This means that for some $s\ge r$, $M$ can be obtained by surgery on an $s$-component diagonal link $L$ in $S^3$ with exactly $r$ zero-framed components and with $\mil L\ge k$.  

A presentation for $\pi/\pi_k$ can be calculated along the same lines as in the proof of Lemma~\ref{lem:milnor2}.  First set $G=\pi_1(S^3-L)$, and let $E$ be the free group of rank $s$ generated by the meridians $m_i$ of $L$, the first $r$ of which correspond to the zero framed components of $L$.  As above 
$$
G/G_k\ \cong \ \langle m_1,\dots,m_s \st E_k, [m_i,w_i]\rangle
$$
for any choice of  Milnor words $w_i = \ell_i^k$ representing the longitudes $\ell_i$.   By Lemma~\ref{lem:milnor1}, $w_i\in E_k$, and so the commutator relations can be ignored.   The surgery adds relations $w_i=1$ for $i\le r$ and $w_im_i^{n_i}=1$ for $i>r$, where the $n_i$ are the non-zero framings of the last $s-r$ components of $L$.  Since all $w_i\in E_k$, we conclude that
$$
\pi/\pi_k\ \cong \ \langle m_1,\dots,m_s \st E_k, m_{i}^{n_{i}} \ \textup{for}\ i>r\rangle.
$$

Now killing the $m_i$ for $i>r$ defines a surjection $\psi:\pi/\pi_k\to F/F_k$, 
where $F$ is free on the $m_i$ for $i\le r$, that induces an isomorphism on $H_1/$Torsion.  Therefore the composition
$$
\phi = \psi\circ\theta_k:F/F_k\to \pi/\pi_k \to F/F_k
$$
induces an isomorphism on $H_1$.   Since $F/F_k$ is nilpotent, $\phi$ is surjective\foot{This is a well-known property of nilpotent groups. The idea is that if say $\{\phi(x_i)\}$ generates the target $F/F_2$ then any commutator can be written as a product of conjugates of terms of the form $[a\phi(x_i),b\phi(x_j)]^{\pm 1}$ where $a,b\in F_2$. Modulo $F_3$ the conjugations and the $a$ and $b$ can be ignored using basic properties of commutators. In this way it is inductively shown that $F_{k}/F_{k+1}$ is generated by a set of $k^{th}$-order commutators in the set $\{\phi(x_i)\}$. Thus $\{\phi(x_i)\}$ generates $F/F_k$ for any $k$.}  Moreover $F/F_k$ is Hopfian, since it is finitely generated nilpotent, and so $\phi$ is in fact an isomorphism.  This implies that $\theta_k$ is a monomorphism.
\end{proof}

\begin{corollary}\label{cor:connsum} If $M= N\#R$ where $H_1N$ is torsion free and $R$ is a rational homology sphere, then $\mil{M}\leq \mil{N}$.
\end{corollary}

\vskip -.1in

\begin{proof} Suppose that $M$ has Milnor degree $d$ and first Betti number $r$.  Let $G=\pi_1N$, \break $P=\pi_1R$, $\pi=\pi_1M$ and $F$ be the free group of rank $r$. Then $\pi\cong G*P$ so in particular there is a map $i:G\to \pi$.

Express $N$ as zero surgery on a link $L$ in a homology sphere $\Sigma$. Choose a meridional map \break $j:F\to G$. Then the composition $\theta=i\circ j:F\to \pi$ induces an isomorphism on $H_1/$Torsion.  Thus for any $k<d$, the composition
$
F/F_k\stackrel{j_k}\to G/G_k\stackrel{i_k}\to \pi/\pi_k
$
is a monomorphism, by Theorem~\ref{thm:mainextension}.  It follows that $j_k$ is a monomorphism.  But we saw in the proof of Lemma~\ref{lem:milnor2} that $j_k$ is always an epimorphism, and so it is in fact an isomorphism.  Therefore $\mil{L}\geq d$, and so by Theorem~\ref{thm:main}, $\mil{N} =\mil{L}\geq d = \mil{M}$ as claimed.
\end{proof}

We conclude with a strengthening of the realization result, Corollary \ref{cor:realization1}, by proving the existence of $3$-manifolds of arbitrary Milnor degree $\mu>1$ with given first homology of rank $r>2$, or $r=2$ with three exceptions.  Note that we do not address the case $\mu=1$ because of the number theoretic issues raised in \S1.

\begin{theorem}\label{thm:realization1}
For any finitely generated abelian group $A$ of rank $r\ge2$ and any $d\ge2$ $($with the exception of $2,4$ or $6$ when $r=2)$ there exists a $3$-manfold $M$ with $H_1(M)=A$ and Milnor degree $d$. 
\end{theorem}

\vskip  -.1in
\begin{proof}  Fix $r\ge2$ and $d\ge2$, except $d\ne2,4$ or $6$ when $r=2$.  By Theorem~\ref{thm:main} and Corollaries~\ref{cor:torfree} and \ref{cor:realization1}, there exists a zero framed diagonal link $L$ in a homology sphere $\Sigma$ such that $\mil L =d$ and $H_1(\Sigma_L)=\bz^r$.  Set $N =  \Sigma_L$ and $R = L(n_1,1)\#\cdots\#L(n_k,1)$, where $A = \bz^r\oplus \bz_{n_1}\oplus\cdots\oplus\bz_{n_k}$.  Then $M = N\# R$ is surgery on the disjoint union of $L$ with a $k$-component unlink
and so
$$
\mil{M}\ge \mil{L} = d.
$$
But by Corollary~\ref{cor:connsum} and Theorem~\ref{thm:main},
$$
\mil{M}\le \mil{N} = \mil L = d. 
$$
Hence $\mil M=d$.
\end{proof}

\vskip-.3in
\vskip-.1in

\section{Torsion and quantum invariants}   

It was noted in the introduction that every integral homology sphere has infinite Milnor degree.  In contrast, there exist rational homology spheres of arbitrary Milnor degree, as will be seen below using quantum topology techniques.  In fact, the same techniques will yield examples of $3$-manifolds with any prescribed Milnor degree and first Betti number, complementing the realization results of the previous section.

\part{Quantum $p$-orders}

The main result in \cite{CM1} relates the Milnor degree $\mil{M}$ of any 3-manifold $M$ with its {\itb\boldmath quantum $p$-order} $\o_p(M)$ (we assume that the reader is familiar with \cite{CM1}, and adopt the notation used there) and its mod $p$ first Betti number 
$$
b_p(M) = \rk H_1(M;\bz_p)
$$ 
for any prime $p\ge5$.   For notational economy we use a rescaling $\oh_p$ of $\o_p$ by dividing by $\o_p(S^1\times S^2) = (p-3)/2$.  Then by \cite[\S4.3]{CM1},
$$
\oh_p \ge b_p (\mu-1)/(\mu+1).
$$   
Solving for $\mu$ yields a useful upper bound for the Milnor degree:

\begin{theorem}\label{thm:porder}
For any prime $p>3$, the Milnor degree $\mu$ of $3$-manifolds satisfies the inequality
$
\mu  \le (b_p+\oh_p)/(b_p-\oh_p)
$
where $b_p$ and $\oh_p$ are as defined above.
\hfill$\square$
\end{theorem}

To apply this result, one must restrict to $3$-manifolds whose $p$-orders can be calculated, or at least estimated. Among these are the manifolds obtained by surgery on iterated Bing doubles of the Hopf link, whose $p$-orders are computable using the methods of \cite{CM1} as shown below. 

\part{Bing doubling} 

The (untwisted) {\itb Bing double} of a link $L$ along one of its components $K$ is obtained from $L$ as follows:  First add a $0$-framed pushoff $K'$ of $K$, and then replace $K\cup K'$ with a pair $K_1\cup K_2$ of linked unknotted components as shown in Figure~\ref{fig:bing}.

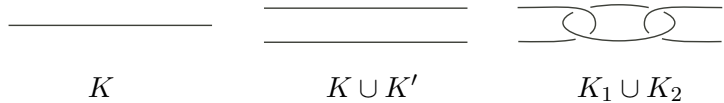
\begin{figure}[h!]\label{fig:bing}
\setlength{\unitlength}{1pt}
\begin{picture}(280,35)
\put(0,20){\includegraphics[height=15pt]{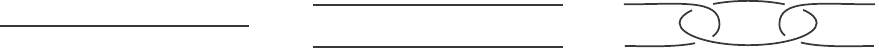}}
\put(30,0){$K$}  
\put(120,0){$K\cup K'$}  
\put(215,0){$K_1\cup K_2$}  
\newline\newline
\end{picture}
\caption{Bing doubling}
\end{figure}

To understand how the quantum invariants of surgery on a (framed) link are affected by Bing doubling, we shall appeal to the following quantum calculation:

\vskip.1in

\begin{lemma}\label{lem:porder}
If $L$ is a framed link with a $0$-framed component $K$, and $L'$ is the Bing double of $L$ along $K$ with both new components $0$-framed, then $\oh_p(S^3_{L'}) = \oh_p(S^3_L) + 1$.
\end{lemma}

\proof
By equation (21) in \cite{CM1}, it suffices to show 
$
\o_p\br{L'} = \o_p\br L + p-3
$ 
where
$$
\br L = \sum_{k<p/2} (a,k]\,J_{L,k}
$$
is the $p$-bracket of $L$ (see  \cite[\S1]{CM1} for the definitions of the framed quantum integers $(a,k]$ and the colored Jones polynomials $J_{L,k}$).   Here $a$ and $k>0$ are multi-indices of integers, specifying respectively the framings and colorings on the components of $L$.  

Allowing colorings from the group ring $\Lambda_p\bz$, as explained in \cite[\S5]{CM1}, the $p$-bracket $\br L$ can be written as a single colored Jones polynomial $J_{L,\lambda}$ for a suitable multi-index $\lambda$ of elements in $\Lambda_p\bz$, and similarly $\br{L'} = J_{L',\lambda'}$.  In particular $K$, $K_1$ and $K_2$ should all be colored with 
$$
\omega = \sum_{k<p/2} [k]\,k
$$
since they are $0$-framed,  and so setting $L_0=L-K$ we have
$(L,\lambda) = (L_0\cup K,\lambda_0\cup\omega)$ and $(L' ,\lambda' ) =  (L_0\cup K_1\cup K_2,\lambda_0\cup\omega\cup\omega)$.  In fact there is an alternative color that can be used for any (or all) of the $0$-framed components, namely
$$
\omega' = \sum_{\odd k<p} [k]\,k.
$$
For example, $J_{L_0\cup K,\lambda_0\cup\omega} = J_{L_0\cup K,\lambda_0\cup\omega'}$.  This is a consequence of the ``symmetry principle" established in \cite[\S4]{KM}.  

Now following \cite{CM1}, we say that two $\Lambda_p\bz$-colored framed links $(L_i,\lambda_i)$ (for $i=1,2$) are {\it equivalent}, written $(L_1,\lambda_1) \approx (L_2,\lambda_2)$, provided $J_{L_1,\lambda_1} = J_{L_2,\lambda_2}$, and extend this to an equivalence relation on the set of $\Lambda_p$-linear combinations of $\Lambda_p\bz$-colored framed links.  We also consider the notion of {\it weak equivalence} $(L_1,\lambda_1) \sim (L_2,\lambda_2)$, defined by the condition $\o_p(J_{L_1,\lambda_1}) = \o_p(J_{L_2,\lambda_2})$.  

Noting that $\o_p(J_{L\sqcup\raisebox{1pt}{$\scriptscriptstyle\bigcirc$},\lambda\cup\omega})=\o_p(J_{L,\lambda}) + p-3$, where $\sqcup\medcirc$ denotes the distant union with an unknot, it suffices to show 
$
(L',\lambda') \sim (L\sqcup\medcirc,\lambda\cup\omega).
$
This is seen by a sequence of (weak) equivalences.  First observe that
\begin{eqnarray*}
(L',\lambda') \ &=& \ \sum_{k<p/2}\ (L_0\cup K_1\cup K_2,\lambda_0\cup [k]k \cup\omega) \\
\ &\approx&\  \sum_{k<p/2}\ (L_0\cup K\cup K'\sqcup\medcirc,\lambda_0\cup k\cup k\cup\omega) \\
\ &\approx& \ \sum_{k<p/2}\sum_{\stackrel{\ \odd}{j<p}} (L_0\cup K\sqcup\medcirc,\lambda_0\cup j\cup\omega) \\
\ &=& \ \displaystyle\sum_{\odd j<p}{\textstyle\frac12}(p-j)\,(L_0\cup K\sqcup\medcirc,\lambda_0\cup j\cup\omega) \\
\end{eqnarray*}
where the first equivalence is a special case of equation (24) in \cite[\S5]{CM1}, the second follows from a well known cabling principle (see for example \cite[\S3.10]{KM}), and the last equality holds since each odd $j$ occurs exactly $\frac12(p-j)$ times in the double sum.    
These equivalences are illustrated below:

\begin{figure}[h!]
\setlength{\unitlength}{1pt}
\begin{picture}(310,40)
\put(30,10){\includegraphics[height=25pt]{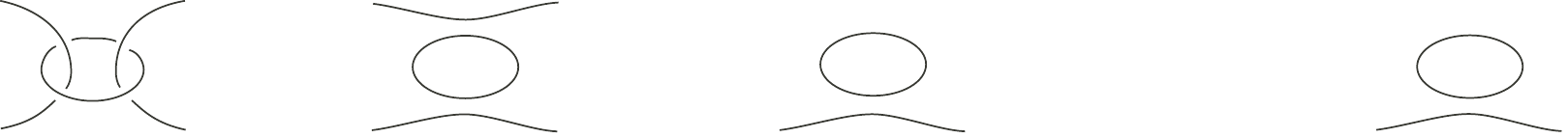}}
\put(0,20){$\sum_k [k]$}  
\put(43,10){\small $\omega$}
\put(52,35){\small $k$}
\put(65,20){$\ \approx \ \sum_k$}  
\put(112,35){\small $k$}
\put(112,5){\small $k$}
\put(127,20){\small $\omega$}
\put(130,20){$\ \ \,\approx \ \sum_{k,j}$}
\put(190,5){\small $j$}
\put(200,30){\small $\omega$}
\put(212,20){$= \ \sum_{j}\frac12(p-j)$}
\put(300,5){\small $j$}
\put(310,30){\small$\omega$}
\end{picture}
\end{figure}

Since $[j] = j + O(h)$ and $p=O(h)$ (in fact $O(h^{p-1})$), it follows that
$$
(p-j)/2 = m[j] + O(h)
$$
where $m=(p-1)/2$.  Thus the last sum is weakly equivalent to
$$
m \sum_{\odd j<p} (L_0\cup K\sqcup\medcirc,\lambda_0\cup [j]j\cup\omega) \ \ \sim \ \ (L_0\cup K\sqcup\medcirc,\lambda_0\cup\omega'\cup\omega) \\
$$
since $\o_p(m)=0$, which equals $(L\sqcup\medcirc,\lambda\cup\omega)$ by the ``symmetry principle" alluded to above. This completes the proof of the lemma.
\qed

\part{Milnor degrees of Bing double surgeries}

We now have the tools to calculate the Milnor degrees of many surgeries on iterated Bing doubles of the Hopf link $H$. In particular, let $H^d$ be the $(d-1)^\textup{st}$ iterated double of $H$ (so $H^1=H$, $H^2$ is the Borromean rings, etc.) and consider the associated framed link $L(n_0,\dots,n_d)$, where the $n_i$ are the framings.  Set 
$
M(n_0,\dots,n_d) = S^3_{L(n_0,\dots,n_d)},
$ 
the $3$-manifold obtained by surgery on $L(n_0,\dots,n_d)$.

\begin{proposition}\label{prop:bing} If $d>1$ and $n_0,\dots, n_d$ have a common prime factor $p>3$, then $M(n_0,\dots,n_d)$ has Milnor degree $d$.
\end{proposition}


\proof
It is well known that $\mil {H^d} =d$ (see \cite[Theorem 8.1]{C4}) and so 
$$
\mil {M_d(n_0,\dots,n_d)}\ge d.  
$$

For the reverse inequality, we apply Theorem~\ref{thm:porder}.  Clearly $M(n_0,\dots,n_d)$ has mod $p$ first Betti number $d+1$, the number of components in $H^d$, since the pairwise linking numbers of $H^d$ vanish when $d>1$. Furthermore, its $p$-order is the same as the $p$-order of the corresponding $0$-surgery $M(0,\dots,0)$, since the framed quantum integers $(a,k]$ depend only on $a\pmod p$.  But $M(0,0)=S^3$ has $p$-order $0$, and so $M(n_0,\dots,n_d)$ has $p$-order $d-1$, by repeated application of Lemma~\ref{lem:porder}.  Thus  
$$
\mil {M(n_0,\dots,n_d)} \ \le \ \frac{(d+1)+(d-1)}{(d+1)-(d-1)} = d.
$$  
This completes the proof.  \qed

\begin{remark}
\textup{This proposition fails in general when $d=1$.   Indeed $M(p,q)$ is $(p,q)$-surgery on $H$, or equivalently $p-(1/q)$ surgery on the unknot, which is just the lens space $L(pq-1,q)$.  By the calculations in section 1 we see, for example, that $M(5,5)=L(24,5)$ has Milnor degree greater than 1, whereas $M(7,7)=L(48,7)$ has Milnor degree 1.}
\end{remark}

\part{Realization}

\begin{theorem}\label{thm:realization2}
For any integers $b\ge0$ and $d\ge1$ there exist $3$-manifolds with first Betti number $b$ and Milnor degree $d$.
\end{theorem}

\begin{proof}
The connected sum of $b$ copies of $S^1\times S^2$ (which is zero surgery on a $b$-component unlink) has infinite Milnor degree, taking care of the case when $d$ is infinite (cf.\ the introduction).  So assume $d$ is finite.

First consider $d>1$.  If $b=0$, simply apply the previous proposition with all $n_i>0$.  For example the rational homology spheres $M(5,\dots,5)$ realize all possible Milnor degrees $>1$.  If $b>0$, then write $b=q(d+1)+r$ with $0\le r\le d$, and define $M_d=M(0,\dots,0)$ and $N_r=M(0,\dots,0,5,\dots,5)$, each manifold with $d+1$ entries (i.e.\ surgery on $H^d$) with the latter having $r$ zeros followed by $d+1-r$ fives.  Now set 
$$
M = M_d\ \#\ \cdots\ \#\ M_d\ \#\ N_r
$$
with $q$ copies of $M_d$, which clearly has first Betti number $b$.   

Since $M$ is surgery on a disjoint union $L$ of copies of $H^d$, and $\mil{L}=d$, it has Milnor degree $\mil{M}\ge d$.   Evidently 
$
b_5(M) = (q+1)(d+1)
$ 
and arguing as in Proposition~\ref{prop:bing} (noting that $\oh_p$ multiplies under connected sums) we have $\oh_5(M) = (q+1)(d-1) < b_5(M)$.  Thus 
$$
\mil{M}\le \frac{2d(q+1)}{2(q+1)} = d
$$
by Theorem 4.1, and so in fact $\mil{M}=d$.

The case $d=1$ is handled by a separate argument.  When $b=0$, simply take a lens space of degree 1, for example $L(5,2)$.  For $b>0$, consider
$$
M = L(5,2)\ \#\ S^1\times S^2\ \#\ \cdots\ \#\  S^1\times S^2
$$
with $b$ copies of $S^1\times S^2$, which clearly has first Betti number $b$.  Also, $M$ has the same torsion linking form as $L(5,2)$, and so $\mil{M}=1$ by Corollary~\ref{cor:deg1}. 
\end{proof}

\appendix

\section{Proof that the Milnor degree is homological}  

It suffices to prove the following:

\begin{lemma*} Let $L$ be a framed link in an integral homology sphere $\Sigma$, and $\Sigma'$ be any other integral homology sphere.  Then there exists a framed link $L'$ in $\Sigma'$ such that \break $(1)$ $\mil {L'} = \mil L$ \ and \ $(2)$ $\Sigma'_{L'} \cong \Sigma_L$.
\end{lemma*}

\vskip-.1in

\begin{proof}  We may assume that $\Sigma'$ is obtained from $\Sigma$ by $\pm1$-surgery on a knot $K$ in $\Sigma$, that is $\Sigma' = \Sigma_K$, since any two integral homology spheres are related by a sequence of such surgeries.

First isotop $K$ in $\Sigma$, possibly crossing $L$ in the process, to arrange that $\mil{K\cup L} = \mil L$.  Proposition A.8 in \cite{HO}, proved using the first author's Theorem 3.3 in \cite{C5}, shows that this can be done.  Now set $L'=K^*\cup L$ where $K^*$ is the 0-framed meridian to $K$ in $\Sigma$.  

The knot $K^*$ is isotopic to $K\subset\Sigma'$ (meaning the core of the surgery) and so $\Sigma' - L' \ = \ \Sigma - (K\cup L)$.
But the Milnor degree of a link depends only on its complement (see e.g. \cite[Lemma A.3]{HO}) and so
$
\mu_{L'} = \mu_{K\cup L} = \mu_L
$
which shows (1).  Also  
$$
\Sigma'_{L'} = \Sigma_{K\cup K^*\cup L} \cong \Sigma_L
$$
since $\Sigma_{K\cup K^*} \cong \Sigma$, which shows (2) and so completes the proof.
\end{proof}

\section{On the number of independent Milnor invariants}  

In this appendix we prove Lemma~\ref{lem:numbertheory}, that the number $M_k^r$ of linearly independent Milnor invariants of degree $k$ for $r$-component links in $S^3$ with vanishing lower degree invariants, given by Orr's formula \cite{O1}
$$
M_k^r \ = \ rN_k^r -N_{k+1}^r \qquad\textup{where}\qquad N_k^r \ = \ \frac1k \, \sum_{d|k}\,\mu(d)\,r^{k/d}\ ,
$$
is {\it positive} except when $(r,k) = (2,2), (2,4)$ or $(2,6)$.  All other values of $(r,k)\ge(2,2)$ will be referred to as {\it generic}.  Note that $M_k^r=0$ in the three {\it exceptional} cases, as seen by computing the first six values of $N_k^2 = 1,2,3,6,9,18$, for $k=2,3,4,5,6,7$.  

We propose to show for generic $(r,k)$ that the numbers $N_k^r$ satisfy the bounds
$$
\frac{r^k}{k+1} \ \le \ N_k^r \ < \ \frac{r^k}k. \put(-250,0){$(*)$}
$$
In fact the upper bound still holds in the exceptional cases, by the computation above, while the lower bound fails.  But this is enough to see that in the generic case
$$
N_{k+1}^r \ < \ \frac{r^{k+1}}{k+1} \ = \ r\frac{r^k}{k+1} \ \le \ rN_k^r
$$
and so $M_k^r>0$ as desired.  

It remains to establish the bounds $(*)$.  The argument is based on the following:

\begin{lemma*}
For any integers $r,k\ge2$ set
$$
\calp(r^k) \ = \ \sum_{p|k} \,r^{k/p}
$$
where the sum is over the distinct prime divisors $p$ of $k$.  Then $r^k>\calp(r^k)$.  In fact, $r^k\ge (k+1)\calp(r^k)$ except when $(r,k) = (2,2), (2,4)$ or $(2,6)$.
\end{lemma*}

\begin{proof}
First note that each term in $\calp(r^k)$ is bounded above by $r^{k/2}$, and so
$$
\calp(r^k) \ \le \ \omega r^{k/2}
$$
where $\omega$ is the number of distinct prime factors of $k$.  Therefore, to prove the first inequality it suffices to show $r^{k/2}  > \omega$, and this is easy:  Let $p_i$ denote the $i^\textup{th}$ prime, starting with $p_1=2$.  Then $r^{k/2}\ge r^{p_2\cdots p_{\omega}} \ge r^{p_{\omega}} > r^{\omega} > \omega$.

To prove the second inequality it would suffice as above to show $$r^{k/2} > (k+1)\omega.$$  This is in fact true for any $(r,k)$ with $k\ge7$.  For in this case $\omega\le k/5$, and so it is enough to verify the inequality $2^{k/2} > (k+1)k/5$ for $k\ge7$, which is straightforward by comparing derivatives of both sides.  

For $k\le6$, it is easiest to test the desired inequality $r^k\ge (k+1)\calp(r^k)$ by direct calculation.  In particular for $k=2,3,4,5$ and $6$, in turn, it reduces to $r\ge3$, $r^2\ge4$, $r^2\ge5$, $r^4\ge6$ and $r^4\ge7+7r$, and therefore fails only for $(r,k) = (2,2), (2,4)$ or $(2,6)$.  This completes the proof.
\end{proof}

We now prove $(*)$.  For any integer $n$, let $\omega(n)$ denote the number of distinct prime factors of $n$.  Fix $r$ and $k$, and set $\omega=\omega(k)$ as above.  Then $N_k^r$ can be written as an alternating sum
$$
N_k^r \ = \  \frac1k \, \sum_{s=0}^\omega \ (-1)^s\,n_s
$$
where $n_s$ collects the terms $r^{k/d}$ in $N_k^r$ associated with divisors $d$ of $k$ that are products of $s$ distinct primes, i.e.\
$$
n_s \ = \!\!\! \sum_{\stackrel{\scriptstyle d|k \,,\, \omega(d) = s}{\mu(d)\ne0}} \!\!\! r^{k/d}.
$$
In particular $n_0 = r^k$ and $n_1 = \calp(r^k)$, in the notation of Lemma B, and so for generic $(r,k)$ we have $n_0>(k+1)n_1$, or equivalently $n_0-n_1 \ge kn_0/(k+1)$.  Furthermore, for any $s<\omega$ we have
$$
n_s \ > \!\! \sum_{\stackrel{\scriptstyle d|k \,,\, \omega(d) = s}{\mu(d)\ne0}} \!\!\! \calp(r^{k/d}) \ \ge \ (s+1)n_{s+1} \ \ge \ n_{s+1}
$$
where the second inequality follows from the observation that each term in the sum defining $n_{s+1}$ appears in exactly $s+1$ terms $\calp(r^{k/d})$ in the displayed sum.  Therefore $n_0,\dots,n_{\omega}$ is a {\it decreasing} positive sequence with $N_k^r = (n_0-n_1+-\cdots \pm n_{\omega})/k$, and so
$$
\frac{r^k}{k+1} \ = \ \frac{n_0}{k+1} \ \le \ \frac{n_0-n_1}{k} \ \le \ N_k^r \ < \ \frac{n_0}k \ = \ \frac{r^k}k
$$
as desired.

\vskip-.2in

\bibliographystyle{plain}
\bibliography{mybib}
\end{document}